\documentclass[11pt,a4paper]{amsart}
\usepackage[utf8]{inputenc}
\usepackage[english]{babel}
\usepackage{amsmath}
\usepackage[psamsfonts]{amssymb}
\usepackage{mathrsfs}
\usepackage{amsthm}
\usepackage{color}
\usepackage{verbatim,euscript,eufrak}
\usepackage{caption}
\usepackage{graphicx}
\def\dep{D(h_\ep(t),\ep)}
\def\depf{\overline{D}(h_\ep(t),\ep)}
\def\psiep{\psi_\ep}

\def\R{\mathbb{R}}
\def\al{\alpha}
\DeclareMathOperator\dive{div}
\DeclareMathOperator\curl{curl}

\def\omep{\curl\vep}
\newcommand\nl[2]{\|#2\|_{L^{#1}}}
\newcommand\nll[2]{\bigl\|#2\bigr\|_{L^{#1}}}
\newcommand\nh[2]{\|#2\|_{H^{#1}}}
\newcommand\lip[1]{\|#1\|_{W^{1,\infty}}}
\def\vep{v_\ep}
\def\vepk{v_{\ep_k}}

\def\etaep{\eta_\ep}

\def\fep{f_\ep}
\def\fept{\widetilde{f}_\ep}

\def\ale{\al_\ep}
\def\alek{\al_{\ep_k}}
\def\toep{\stackrel{\ep\to0}{\longrightarrow}}
\def\toepk{\stackrel{\ep_k\to0}{\longrightarrow}}
\def\phiep{\varphi_\ep}
\def\phiepk{\varphi_{\ep_k}}
\def\xiep{\Xi_\ep}
\def\xiepk{\Xi_{\ep_k}}
\def\Lip{W^{1,\infty}}
\def\loc{_{\text{loc}}}

\newcommand\Om{\Omega}
\def\slr{\Omega_\ep}
\def\flr{\RR^2\setminus\Omega_\ep}

\newcommand{\RR}{\mathbb R}

\newcommand{\pat}{\partial_t}

\newcommand{\na}{\nabla}
\newcommand{\ue}{u_{\varepsilon}}
\newcommand{\ep}{\varepsilon}

\newcommand{\wep}{\widetilde{u}_{\ep}}

\newcommand{\norm}[1]{\lVert#1\rVert}
\newcommand{\normm}[1]{\bigl\lVert#1\bigr\rVert}
\newcommand{\vertiii}[1]{{\left\vert\kern-0.25ex\left\vert\kern-0.25ex\left\vert #1 
    \right\vert\kern-0.25ex\right\vert\kern-0.25ex\right\vert}}

\newcounter{comentcount}
\newcounter{teocount}
\newtheorem{lem}{Lemma}
\newtheorem{prop}{Proposition}

\newtheorem{teo}[teocount]{Theorem}

\title[]{A small solid body with large density in a planar fluid is negligible}

\author{Jiao He}
\author{Dragoș Iftimie}
\begin{document}

\begin{abstract}
In this article, we consider a small rigid body moving in a viscous fluid filling the whole $\RR^2$. We assume that the diameter of the rigid body goes to 0, that the initial velocity has bounded energy and that the density of the rigid body goes to infinity. We prove that the rigid body has no influence on the limit equation by showing convergence of the solutions towards a solution  of the Navier-Stokes equations in the full plane $\mathbb{R}^{2}$.  
\end{abstract}

\maketitle

\section{Introduction}
In this paper, we consider a fluid-solid system consisting in a small smooth rigid body $\Om_\ep$ of size $\ep$ evolving in a viscous fluid filling the whole of $\RR^2$. Our aim is to determine the limit of this coupled system when the size of the rigid body $\ep$ goes to 0.

Let us describe now the fluid solid system of equations. To do that, we need to introduce some notation. We denote by $u_\ep$, respectively $p_\ep$, the velocity, respectively the pressure, of the fluid; they are defined on $\flr$, the exterior of the smooth rigid body $\slr$. The evolution of the rigid body $\slr(t)$ is described by $h_\ep$, the position of its center of mass, and by $\theta_\ep$, the angle of rotation of the rigid body compared with the initial position. We have that
\begin{equation*}
  \slr(t)=h_\ep(t)+ \begin{pmatrix}
  \cos \theta_{\ep}(t) & -\sin \theta_{\ep}(t)\\
  \sin \theta_{\ep}(t) & \cos  \theta_{\ep}(t)
  \end{pmatrix}
\bigl(\slr(0)-h_\ep(0)\bigr).
\end{equation*}

The velocity of the fluid verifies the incompressible Navier-Stokes equations in the exterior of the rigid body:
\begin{equation}{\label{ns-ob1}}
\frac{\partial u_{\ep}}{\partial t} +u_{\ep} \cdot \nabla u_{\ep}- \nu \Delta u_{\ep} + \nabla p_{\ep} =0,\quad \dive u_\ep=0\quad \text{for}\; t>0\text{ and }x \in \flr(t).
\end{equation}
On the boundary of the rigid body we assume no-slip boundary conditions:
\begin{equation}
 u_{\ep}(t,x)= h_{\ep}'(t)+ \theta_{\ep}'(t)(x-h_{\ep}(t))^{\bot} \; \text{for}\; t >0 \text{ and }x \in \partial\slr(t).
\end{equation}
Moreover, the velocity is assumed to vanish at infinity:
\begin{equation}
\lim_{|x|\rightarrow \infty} u_{\ep}(t,x) =0 \; \text{for}\; t \geq0.
\end{equation}

Now we write down the equations of motion of the solid body. Let us denote by $m_\ep$  the mass of the solid and by $J_\ep$ the momentum of inertia of the solid. We also denote by $\sigma(u_\ep,p_\ep) $  the stress tensor of the fluid:
\begin{equation*}
\sigma(u_\ep,p_\ep) =2\nu D(u_\ep)-p_\ep I_2  
\end{equation*}
where $I_2$ is the identity matrix and $D(u_\ep)$ is the deformation tensor
\begin{equation*}
D(u_\ep)=\frac{1}{2}\bigl(\frac{\partial u_{\ep,i}}{\partial x_{j}}+ \frac{\partial u_{\ep,j}}{\partial x_{i}} \bigr)_{i,j}.  
\end{equation*}

Then the solid body $\slr(t)$ evolves according to Newton's balance law for linear and angular momenta:
\begin{equation}
m_{\ep} h''_{\ep}(t) = - \int_{\partial \slr(t)} \sigma(u_\ep,p_\ep) n_\ep  \;\;\; \text{for}\; t >0,
\end{equation}
and
\begin{equation}{\label{solid2}}
J_{\ep}\theta''_{\ep}(t) = - \int_{\partial \slr(t)} (\sigma(u_\ep,p_\ep) n_\ep)\cdot (x-h_{\ep})^{\bot}  \; \text{for}\; t >0.
\end{equation}

Above $n_\ep$ denotes the unit normal to $\partial \slr$ which points to the interior of the rigid body $\slr$, the orthogonal $x^\perp$ is defined by $x^\perp=(-x_2,x_1)$ and $\sigma(u_\ep,p_\ep) n_\ep$ denotes the matrix $\sigma(u_\ep,p_\ep)$ applied to the vector $n_\ep$.

One can obtain energy estimates for this system of equations. If we formally multiply the equation of $u_\ep$ by $u_\ep$, do some integrations by parts using also the equations of motion of the rigid body, we get the following energy estimate:
\begin{multline}\label{sfenerg}
\|u_\ep(t)\|_{L^2(\flr)}^2 + m_\ep |h'_\ep(t)|^2 + J_\ep |\theta'_\ep(t)|^2 + 4\nu\int_0^t \|D(u_\ep)\|_{L^2(\flr)}^2 \\
\leq \|u_\ep(0)\|_{L^2(\flr)}^2 + m_\ep |h'_\ep(0)|^2 + J_\ep |\theta'_\ep(0)|^2.
\end{multline}

To solve the system of equations \eqref{ns-ob1}--\eqref{solid2}, we need to impose the initial data. For the fluid part of the system we need to impose the initial velocity $\ue(0,x)$. The two equations describing the evolution of the rigid body are second-order in time, so we need to know $h_\ep(0)$, $h'_\ep(0)$, $\theta_\ep(0)$ and $\theta'_\ep(0)$. The system of equations being  translation invariant, we can assume without loss of generality that the initial position of the center of mass of the rigid body is in the origin: $h_\ep(0)=0$. Moreover, from the definition of the angle of rotation $\theta_\ep$ we obviously have that $\theta_\ep(0)=0$. So we only need to impose $\ue(0,x)$, $h'_\ep(0)$ and $\theta'_\ep(0)$. The initial velocity will be assumed to be square integrable only. As such, its trace on the boundary is not well-defined. Only its normal trace is defined thanks to the divergence free condition. Therefore, we need to impose the following compatibility condition on the initial velocity: 
\begin{equation}\label{compat}
 \ue(0,x)\cdot n_\ep= \bigl[h'_{\ep}(0)+\theta'_{\ep}(0)\big(x-h_{\ep}(0)\big)^{\perp}\bigr]\cdot n_\ep\quad\text{on }\partial\slr(0). 
\end{equation}
In conclusion, to solve the system of equations \eqref{ns-ob1}--\eqref{solid2}, we need to impose  that $\ue(0,x)\in L^2(\R^2\setminus\slr(0))$, that $\dive \ue(0,x) =0$ in $\R^2\setminus\slr(0)$ and the compatibility condition \eqref{compat}. There is no condition required on  $h'_\ep(0)$ and $\theta'_\ep(0)$ while $h_\ep(0)=0$ and $\theta_\ep(0)=0$.

To state the classical result of existence and uniqueness of solutions of \eqref{ns-ob1}--\eqref{solid2}, it is practical to extend  the velocity field $\ue$ inside the rigid body as follows:
\begin{equation}\label{extension}
\wep(t,x)=
\begin{cases}
\ue(t,x)&\quad\text{if }x\in\R^2\setminus\slr(t)\\
h'_{\ep}(t)+\theta'_{\ep}(t)(x-h_{\ep}(t) )^{\perp} &\quad\text{if }x\in\slr(t). 
\end{cases}
\end{equation}

The conditions imposed on the initial data ensure that $\wep(0,x)$ belongs to $L^2(\R^2)$ and is divergence free in $\R^2$. 

Let us denote by $\rho_\ep$ the density of the rigid body $\slr$. We extend $\rho_\ep$ in the fluid region $\R^2\setminus\slr$ by giving it value 1:
\begin{equation*}
 \widetilde{\rho}_{\ep}(t,x) =
   \begin{cases}
   1,& x\in \R^2\setminus\slr(t) \\
    \rho_{\ep}, & x\in \slr(t).
   \end{cases}
\end{equation*}

Due to the energy estimates \eqref{sfenerg}, global existence of finite energy solutions of  \eqref{ns-ob1}--\eqref{solid2} have been proved in a variety of settings. The literature is vast, we give here just a few references dealing with the dimension two: in \cite{desjardins_existence_1999}, \cite{hoffmann_motion_1999} and \cite{san_martin_global_2002} the authors consider the case of one or several rigid bodies moving in a bounded domain filled with a viscous fluid while in \cite{takahashi_global_2004} the authors consider a single disk moving in a fluid filling the whole plane. The existence for the problem we are considering here was not explicitly studied in these works (because we do not assume the rigid body to be a disk), but more complicated cases have been considered in the literature: the case of a 2D bounded domain where collisions with the boundary must be taken into account (see \cite{desjardins_existence_1999}, \cite{hoffmann_motion_1999} and \cite{san_martin_global_2002}) and the case of $\R^3$ with a rigid body of arbitrary shape (see for example \cite{yudakov_solvability_1974} and \cite{serre_chute_1987}). From these results we can extract the following statement about the existence and uniqueness of solutions of  \eqref{ns-ob1}--\eqref{solid2}.  We use the notation $\R_+=[0,\infty)$ and emphasize that the endpoint 0 belongs to $\R_+$. This is important when we write local spaces in $\R_+$ like for instance $L^2\loc(\R_+)=\{f\ ; \ f\text{ square integrable on any interval }[0,t]\}$. We will give a formulation of the PDE in terms of the extended velocity $\wep$.
\begin{teo}\label{DesEst}
Let $\ue(0,x)\in L^2(\flr(0))$ be divergence free and verifying the compatibility condition \eqref{compat}. We assume that $h_\ep(0)=0$ and $\theta_\ep(0)=0$ and we extend $\ue(0,x)$ to $\wep(0,x)$ as in \eqref{extension}. Then $\wep(0,x)$ is divergence free and square integrable on $\R^2$ and there exists a unique global weak solution $(\ue, h_{\ep}, \theta_{\ep})$ of \eqref{ns-ob1}--\eqref{solid2} in the following sense:
\begin{itemize}
\item $\ue, h_\ep, \theta_\ep$ satisfy
\begin{equation*}
\ue\in L^{\infty}(\R_+; L^{2}(\flr)) \cap L^{2}\loc(\R_+; H^{1}(\flr)),
\end{equation*}
\begin{equation*}
h_\ep \in W^{1,\infty}(\R_+; \R^2),\quad \theta_\ep \in W^{1,\infty} (\R_+; \R);
\end{equation*}
\item if we define $\wep$ as in \eqref{extension} then  $\wep$ is divergence free with $D\wep(t,x)=0$ in $\slr(t)$ and the equations of motion are verified in the sense of distributions under the following form
\begin{equation*}
-\int_{0}^{\infty} \int_{\RR^{2}} \widetilde{\rho}_{\ep} \wep \cdot \big(\pat \varphi_{\ep} + (\wep \cdot \na) \varphi_{\ep} \big) + 2 \nu \int_{0}^{\infty} \int_{\RR^{2}} D (\wep) : D (\varphi_{\ep}) 
 = \int_{\RR^{2}} \widetilde{\rho}_{\ep}(0) \wep(0) \cdot \varphi_{\ep}(0). 
\end{equation*}
for any divergence free test function $\varphi_{\ep} \in H^1(\R_+\times\R^2)$ compactly supported in time and such that $D\varphi_\ep(t,x)=0$ in $\slr(t)$; 
\end{itemize}
Moreover, $\wep$ satisfies the following energy inequality:
\begin{equation}{\label{energyineq}}
\int_{\RR^{2}} \widetilde{\rho}_{\ep} |\wep|^{2}  + 4 \nu \int_{0}^{t} \int_{\RR^{2}} |D (\wep)|^{2} \leq \int_{\RR^{2}} \widetilde{\rho}_{\ep}(0) |\wep(0)|^{2} \;\;\;\; \forall t>0.
\end{equation}
\end{teo}

As mentioned before, we are interested in describing the asymptotic behavior of this fluid-solid system when the diameter of the rigid body $\slr$ goes to 0. There are several papers dealing with this issue when the rigid body does not move with the fluid. 
Iftimie, Lopes Filho and Nussenzveig Lopes \cite{iftimie_two-dimensional_2006} have treated the asymptotic behavior of viscous incompressible 2D flow in the exterior of a small fixed rigid body as the size of the rigid body becomes very small, see also \cite{chipot_limits_2014} for the case of the periodic boundary conditions. Moreover, Lacave \cite{lacave_two-dimensional_2009} considered a two-dimensional viscous fluid in the exterior of a thin fixed rigid body shrinking to a curve and proved convergence to a solution of the Navier-Stokes equations in the exterior of a curve. 

Although we are dealing here only with viscous fluids, let us mention that the case of a perfect incompressible fluid governed by the Euler equations also makes sense and the literature is richer. Let us mention a few results. Iftimie, Lopes Filho and Nussenzveig Lopes \cite{iftimie_two_2003} have studied the asymptotic behavior of incompressible, ideal two-dimensional flow in the exterior of a small fixed rigid body when the size of the rigid body becomes very small. Recently, Glass, Lacave and Sueur \cite{glass_motion_2014-1} have studied the case when the solid body shrinks to a point with fixed mass and circulation and is moving with the fluid. The same three authors  also consider in \cite{glass_motion_2016} the case when the body shrinks to a massless pointwise particle with fixed circulation. In that case, the fluid-solid system converges to the vortex-wave system. In addition, Glass, Munnier and Sueur \cite{glass_point_2018} considered the case of a bounded domain. 

As far as we know, there is only one result dealing with the case of a small rigid body moving in a viscous fluid in dimension two. More precisely, Lacave and Takahashi \cite{lacave_small_2017} considered a small moving disk in a two-dimensional viscous incompressible fluid. They used a fixed-point type argument based on previously known $L^{p}-L^{q}$ decay estimates of the linear semigroup associated to the fluid-solid system (see \cite{ervedoza_long-time_2014}). They proved convergence towards the solution of the Navier-Stokes equations in $\RR^{2}$ under the assumption that the rigid body is a disk of radius $\ep$, that the density $\rho_\ep$ is constant plus some smallness assumptions on the initial data (including the smallness of the $L^2$ norm of the initial fluid velocity). More precisely, their result is the following.
\begin{teo}[\cite{lacave_small_2017}]
 There exists a constant $\lambda_0>0$ such that if 
\begin{itemize}
\item $\ue(0,x)\in L^2(\flr(0))$ is divergence free and verifies the compatibility condition \eqref{compat};
\item the rigid body is the disk $\slr=D(h_\ep,\ep)$;
\item the density $\rho_\ep$ is assumed to be independent of $\ep$;
\item $\wep(0,x)$ converges weakly in $L^2(\R^2)$ to some $u_0(x)$;
\item we have the following smallness of the initial data
\begin{equation}\label{smalllt}
\norm{\ue(0,x)}_{L^2(\flr(0))} +\ep |h'_\ep(0)|+\ep^2 |\theta'_\ep(0)|\leq \lambda_0
\end{equation}
\end{itemize}
then the global solution $\wep$ given by Theorem \ref{DesEst} converges weak$\ast$  in $L^\infty(\R_+;L^2(\R^2))\cap L^2\loc(\R_+;H^1(\R^2))$ towards  the weak solution of the Navier-Stokes equations in $\R^2$  with initial data $u_0$.
\end{teo}

Although they state their result for constant density, presumably the proof can be adapted to the case where $\rho_\ep\geq\rho_0$ for some $\rho_0>0$ independent of $\ep$. On the other hand, the hypothesis that $\slr$ is a disk seems to be essential in the result of \cite{lacave_small_2017}. Indeed, a key ingredient are the estimates of \cite{ervedoza_long-time_2014} and the proof of that result relies heavily on the fact that $\slr$ is a disk because it uses explicit formulae valid only for the case of a disk. Moreover, it is also hard to see how the smallness condition \eqref{smalllt} could be removed in their argument. Indeed, they use a fixed point argument and that requires smallness at some point. Let us observe that in \cite{lacave_small_2017} the authors also obtain uniform bounds in $\ep$ for the velocity of the disk. Therefore, they can prove that the center of mass of the disk converges to some trajectory. However, nothing can be said about this limit trajectory.

Here, we improve the result of \cite{lacave_small_2017} in two respects. First, the rigid body does not need to be a disk. It does not even need to be shrinking homothetically to a point like in \cite{lacave_small_2017}. We only assume that the diameter of the rigid body goes to 0. Second, we require no smallness assumption on the initial fluid velocity $\ue(0,x)$. On the other hand, we need to assume that the density of the rigid body goes to infinity and we are not able to prove uniform bounds on the motion of the rigid body as in \cite{lacave_small_2017}.  More precisely, we will prove the following result.
\begin{teo}{\label{th1}}
We assume  the hypothesis of Theorem \ref{DesEst} and moreover
\begin{itemize}
\item   $ \slr(0) \subset D(0,\ep)$;
\item the mass $m_\ep$ of the rigid body verifies that
\begin{equation}\label{hyprho}
\frac{m_\ep}{\ep^2}\to\infty\quad\text{as }\ep\to0;
\end{equation}
\item $\ue(0,x)$ is bounded independently of $\ep$ in $ L^2(\flr(0))$ and $\sqrt{m_\ep}h'_\ep(0)$ and $\sqrt{J_\ep}\theta'_\ep(0)$ are  bounded independently of $\ep$;
\item $\wep(0,x)$ converges weakly in $L^2(\R^2)$ to some $u_0(x)$ where $\wep(0,x)$ is constructed as in \eqref{extension}.
\end{itemize}
Let $(\ue, h_{\ep}, \theta_{\ep})$  be the global solution  of the system \eqref{ns-ob1}--\eqref{solid2} given by Theorem \ref{DesEst}. Then $\wep$ converges weak$\ast$ in $L^\infty(\R_+;L^2(\R^2))\cap L^2\loc(\R_+;H^1(\R^2))$ as $\ep\to0$ towards the solution of the Navier-Stokes equations in $\R^2$ with initial data $u_0$.
\end{teo}

It will be clear from the proof that the convergence of $\wep$ is stronger than stated. For instance, we shall prove that $\wep$ converges strongly in $L^2\loc$ (see Section \ref{temp}).

Let us remark that if the measure of $\slr$ is of order $\ep^2$ (something which is true if the rigid body  shrinks homothetically to a point, \textit{i.e.} if $\slr(0)$ is $\ep$ times a fixed rigid body) then the hypothesis \eqref{hyprho} means that the density $\rho_\ep$ of the rigid body goes to $\infty$ as $\ep\to0$.

Observe next that the boundedness of $\ue(0,x)$ in $L^2(\flr(0))$ and the boundedness of $\sqrt{m_\ep}h'_\ep(0)$ and $\sqrt{J_\ep}\theta'_\ep(0)$ imply the boundedness of  $\sqrt{\widetilde{\rho}_{\ep}(0,x)}\wep(0,x)$ in $L^2(\R^2)$. Since $\rho_\ep\to\infty$ this implies that $\wep(0,x)$ is bounded in $L^2(\R^2)$. Therefore, the weak convergence of $\wep(0,x)$ to $u_0(x)$ is not really a new hypothesis. 

Moreover, the boundedness of  $\sqrt{\widetilde{\rho}_{\ep}(0,x)}\wep(0,x)$ in $L^2(\R^2)$ and the energy inequality \eqref{energyineq} imply that $\sqrt{\widetilde{\rho}_{\ep}}\wep$ is bounded independently of $\ep$ in the space $L^\infty(\R_+;L^2(\R^2))\cap L^2\loc(\R_+;H^1(\R^2))$. Using again that $\rho_\ep\to\infty$ we deduce that $\wep$ is also bounded independently of $\ep$ in $L^\infty(\R_+;L^2(\R^2))\cap L^2\loc(\R_+;H^1(\R^2))$. And this is all we need to prove the convergence of $\wep$ towards a solution of the Navier-Stokes equations in $\R^2$. Our proof does not require that $\wep$ verifies the boundary conditions on $\slr$, nor do we need that $D\wep=0$ in $\slr$. We only need the above mentioned boundedness of $\wep$ and the fact that it verifies the Navier-Stokes equations (without any boundary condition) in the exterior of the disk $D(h_\ep(t),\ep)$. We state next a more general result.
\begin{teo}\label{mainthm}
Let $\vep$ be a time-dependent divergence free vector field defined on $\R_+\times\R^2$  belonging to the space
\begin{equation}\label{vepspace}
L^{\infty}(\R_+; L^{2}(\RR^{2})) \cap L^{2}\loc(\R_+; H^{1}(\RR^{2})) \cap {C}_{w}^{0}(\R_+; L^{2}\loc(\RR^{2}\setminus \overline{D}(h_\ep(t),\ep)))
\end{equation}
and let $h_\ep\in \Lip(\R_+;\R^2)$.
Assume moreover that
\begin{itemize}
\item $\vep$ is bounded independently of $\ep$ in the above space;
\item $\vep(0,x)$ converges weakly in $L^2$ as $\ep\to0$ to some $v_0(x)$;
\item $\vep$ verifies the Navier-Stokes equations in the exterior of the disk $\overline{D}(h_\ep(t),\ep)$:
  \begin{equation}\label{nsep}
\partial_t\vep-\nu\Delta\vep+\vep\cdot\nabla\vep=-\nabla\pi_\ep \quad \text{in the set } \{(t,x)\ ;\ t>0\text{ and } |x-h_\ep(t)|>\ep\}  
  \end{equation}
for some $\pi_\ep$;
\item the velocity of the center of the disk verifies that $\ep|h_\ep'(t)|\to0$ in $L^\infty\loc(\R_+)$ when $\ep\to0$.
\end{itemize}
Let $v$ be the unique solution of the Navier-Stokes equations in $\R^2$ with initial data $v_0$. Then $\vep$ converges to $v$ as $\ep\to0$ weak$\ast$ in the space $L^{\infty}(\R_+; L^{2}(\RR^{2})) \cap L^{2}\loc(\R_+; H^{1}(\RR^{2}))$. 
\end{teo}
Theorem \ref{mainthm} with $v_\ep=\wep$ implies Theorem \ref{th1}. Indeed, we already observed above that $\wep$ has all the properties required from $\vep$ in Theorem \ref{mainthm}. And the hypothesis made on the mass of the rigid body, see relation \eqref{hyprho}, in Theorem \ref{th1} implies that $\ep|h_\ep'(t)|\to0$ in $L^\infty\loc(\R_+)$ when $\ep\to0$. This can be easily seen from the energy estimate \eqref{sfenerg}. Indeed, the hypothesis of Theorem \ref{th1} implies that the right-hand side of \eqref{sfenerg} is bounded uniformly in $\ep$ so $\sqrt{m_\ep}h_\ep'$ is uniformly bounded in $t$ and $\ep$. The fact that  $\frac{m_\ep}{\ep^2}\to\infty$ and the boundedness of  $\sqrt{m_\ep}h_\ep'$ implies that $\ep h'_\ep\to0$  as $\ep\to0$ uniformly in time.

The idea of the proof of Theorem \ref{mainthm} is completely different from the proof given in \cite{lacave_small_2017}. We multiply \eqref{nsep} with a cut-off vanishing on the disk  $D(h_\ep(t),\ep)$  constructed in a very particular manner. We then pass to the limit with classical compactness methods. The difficulty here is that the cut-off function itself depends on the time, so time-derivative estimates of $v_\ep$ are not so easy to obtain. Also, passing to the limit in the terms $\partial_t v$ and $\Delta v$ is not obvious: the first is difficult because  the time derivative is hard to control and the second one is difficult because the cut-off introduces negative powers of $\ep$ in this term. 

The plan of the paper is the following. In the following section we introduce some notation and prove some preliminary results. In Section \ref{defcutof} we construct the special cut-off near the rigid body. The required temporal estimates are proved in Section \ref{temp}. Finally, we pass to the limit in Section \ref{paslim}.

\section{Notation and preliminary results}
\label{sectnot}
We use the classical notation $C^m$ for functions with $m$ continuous derivatives and $H^m$ the Sobolev space of functions with $m$ square-integrable weak derivatives. The notation $C^m_b$ stands for functions in $C^m$ with bounded derivatives up to order $m$. All function spaces and norms are considered to be taken on $\R^2$ in the $x$ variable unless otherwise specified.  We define $C^\infty_{0,\sigma}$ to be the space of smooth, compactly supported and divergence free vector fields on $\R^2$.  The derivatives are always taken with respect to the variable $x$ unless otherwise specified. The double dot product of two matrices $M=(m_{ij})$ and $N=(n_{ij})$ denotes the quantity $M:N=\sum_{i,j}m_{ij}n_{ij}$. We denote by $C$ a generic universal constant whose value can change from one line to another.

Let  $\varphi\in C^1_b (\R_+;C^\infty_{0,\sigma})$. We define the stream function $\psi$ of $\varphi$ by 
\begin{equation*}
\psi(x)=\int_{\mathbb{R}^{2}} \frac{(x-y)^{\perp}}{2\pi |x-y| ^{2}} \cdot \varphi(y) dy. 
\end{equation*}

It is well-known that $\psi\in C^1_b (\R_+;C^\infty)$ and  $\na^{\perp} \psi = \varphi$. The stream function $\psi$  given above is characterized by two facts. One is that $\nabla^\perp \psi=\varphi$ and another one is that it vanishes at infinity. But in our case, the vanishing at infinity is not important since we will use compactly supported test functions. On the other hand, it is useful to have the stream function small in the neighborhood of the rigid body. We define now a modified stream function, denoted by $\psiep$, which vanishes at the center of the disk $\dep$:
\begin{equation}\label{psiepdef}
\psiep(t,x)=\psi(t,x)-\psi(t,h_{\ep}(t)).
\end{equation}
Observe that even if $\varphi$ is constant in time, the modified stream function still depends on the time through $h_\ep$. We collect some properties of the modified stream function in the following lemma.
\begin{lem}
The modified stream function $\psiep$ has the following properties:
\begin{enumerate}
\item We have that $\psiep\in \Lip (\R_+;C^\infty)$ and $\na^{\perp} \psiep = \varphi$.
\item For all $t,R\geq0$ and $x\in\R^2$ we have that
  \begin{gather}
  \|\psiep(t,\cdot)\|_{L^\infty(D(h_\ep(t),R))}\leq R\nl\infty{\varphi(t,\cdot)}\label{psiep2}\\
\intertext{and}
\|\partial_t\psiep(t,\cdot)\|_{L^\infty(D(h_\ep(t),R))}\leq R\nl\infty{\partial_t\varphi(t,\cdot)}+|h_\ep'(t)|\nl\infty{\varphi(t,\cdot)}\label{psiep3}
  \end{gather}
with the remark that the last relation holds true only almost everywhere in time.
\end{enumerate}
\end{lem}
\begin{proof}
Clearly $\nabla^\perp\psiep=\nabla^\perp\psi=\varphi$. Since $h_\ep\in \Lip (\R_+)$ and $\psi\in C^1_b (\R_+;C^\infty)$ we immediately see that $\psiep\in \Lip (\R_+;C^\infty)$ which proves (i).  

By the mean value theorem
\begin{equation}\label{psiep1}
|\psiep(t,x)|
= |\psi(t,x)-\psi(t,h_{\ep}(t))|
\leq |x-h_\ep(t)|\nl\infty{\nabla\psi(t,\cdot)}
= |x-h_\ep(t)|\nl\infty{\varphi(t,\cdot)}.
\end{equation}
Relation \eqref{psiep2} follows. To prove \eqref{psiep3} we recall that $h_\ep$ is Lipschitz in time so it is almost everywhere differentiable in time. Let $t$ be a time where $h_\ep$ is differentiable. We write
\begin{align*}
\pat\psi_{\ep }(t,x)
&=\pat (\psi(t,x)-\psi(t,h_\ep(t)))\\
&=\pat\psi(t,x)-\pat\psi(t,h_\ep(t))-h_\ep'(t)\cdot\nabla\psi(  t,h_\ep(t))
\end{align*}
so
\begin{align*}
\|\partial_t\psiep(t,\cdot)\|_{L^\infty(D(h_\ep(t),R))}
&\leq \|\pat\psi(t,x)-\pat\psi(t,h_\ep(t))\|_{L^\infty(D(h_\ep(t),R))}  +|h_\ep'(t)|\nl\infty{\nabla\psi(t,\cdot)}\\
& \leq R\nl\infty{\partial_t\nabla\psi(t,\cdot)}+|h_\ep'(t)|\nl\infty{\varphi(t,\cdot)}\\
& = R\nl\infty{\partial_t\varphi(t,\cdot)}+|h_\ep'(t)|\nl\infty{\varphi(t,\cdot)}.  
\end{align*}
This completes the proof of the lemma.
\end{proof}

We will need to define a cut-off function near the rigid body with $L^2$ norm of the gradient as small as possible. This will be done in the next section. For the moment, let us recall that the function that minimizes the  $L^2$ norm of the gradient, that vanishes for $|x|=A$ and is equal to 1 for $|x|=B$ is harmonic. So it is given by the explicit formula
\begin{equation*}
 f _{A,B}:\R^2\to[0,1], \quad  f _{A,B}(x) =
\begin{cases}
0&\text{ if }|x|<A\\
\frac{\ln |x| - \ln A }{ \ln B- \ln A} &\text{ if }A<|x|<B\\
1&\text{ if }|x|>B.
\end{cases}
\end{equation*}
This special cut-off has the following properties.
\begin{lem}\label{f_AB}
We have that $ f _{A,B}\in W^{1,\infty}$. Moreover,
\begin{align*}
\norm{ f _{A,B}(x)-1}_{L^{2}}^{2}&=  \pi A^2\bigl(\frac{\al^2}{2\ln^2\al}-\frac1{2\ln^2\al}-\frac1{\ln\al}\bigr),\\
\norm{\na  f _{A,B}}_{L^{2}} ^{2}&=\frac{2\pi}{\ln\al}  \\
\intertext{and}
\normm{|x|\na^2  f _{A,B}}_{L^{2}(A<|x|<B)} ^{2}&=\frac{4\pi}{\ln\al}  
\end{align*}
where $\al=\frac BA$.
\end{lem}
\begin{proof}
The Lipschitz character of $ f _{A,B}$ is obvious once we remark that $ f _{A,B}$ is smooth for $|x|\neq A$ and $|x|\neq B$ and  continuous across $|x|=A$ and $|x|=B$. 

Next, we have that
\begin{align*}
\norm{ f _{A,B}(x)-1}_{L^{2}}^{2}
&=  \int_{|x|<A}1\,dx+\int_{A<|x|<B}\Bigl| \frac{\ln |x| - \ln B }{ \ln B- \ln A}  \Bigr|^{2} dx\\
&= \pi A^2+ \frac{B^2}{(\ln B- \ln A)^2}\int_{A/B<|y|<1}\ln^2 |y|\,dy\\
&=  \pi A^2\bigl(1+\frac{\al^2}{\ln^2\al}\int_{1/\al}^1\ln^2r\, 2r\,dr\bigr)\\
&= \pi A^2\bigl(\frac{\al^2}{2\ln^2\al}-\frac1{2\ln^2\al}-\frac1{\ln\al}\bigr).
\end{align*}
From the definition of $ f _{A,B}$, we compute for $A<|x|<B$
$$ \na  f _{A,B} = \frac{x}{|x|^{2}\ln\al}\;\;\text{and} \;\; |\na^{2}  f _{A,B}| = \frac{\sqrt2}{|x|^{2}\ln\al}\cdot$$
So
\begin{equation*}
\norm{\na  f _{A,B}}_{L^{2}} ^{2}
=\frac1{\ln^2\al}\int_{A<|x|<B}\frac1{|x|^2}\,dx=\frac{2\pi}{\ln\al}  
\end{equation*}
and
\begin{equation*}
\normm{|x|\na^2  f _{A,B}}_{L^{2}(A<|x|<B)} ^{2}
=\frac2{\ln^2\al}\int_{A<|x|<B}\frac1{|x|^2}\,dx=\frac{4\pi}{\ln\al}.  
\end{equation*}
This completes the proof of the lemma. 
\end{proof}

\section{Cut-off near the rigid body}
\label{defcutof}

We begin now the proof of Theorem \ref{mainthm}. It suffices to prove the following statement.
\begin{prop}\label{prop}
For all finite times $T>0$ there exists a subsequence $\vepk$ which converges weak$\ast$ in $L^\infty(0,T;L^2)\cap L^2(0,T;H^1)$ towards a solution $v\in L^\infty(0,T;L^2)\cap L^2(0,T;H^1)$ of the Navier-Stokes equations on $[0,T)\times\R^2$ with initial data $v_0$.    
\end{prop}

Indeed, let us assume that Proposition \ref{prop} is proved. We know that the Navier-Stokes equations in dimension two have a unique global solution $v$ in the space $L^\infty(\R_+;L^2)\cap L^2(\R_+;H^1)$, see for example \cite{lions_quelques_1969}. The solution $v$ from Proposition \ref{prop} is necessarily the restriction to $[0,T]$ of this unique global solution. Since we have uniqueness of the limit, we deduce that the whole sequence $\vep$ converges weak$\ast$ in $L^\infty(0,T;L^2)\cap L^2(0,T;H^1)$ towards $v$. Since $T$ is arbitrary, Theorem \ref{mainthm} follows.

\bigskip

The rest of this paper is devoted to the proof of Proposition \ref{prop}. Let $T>0$ be fixed. From now on the time $t$ is assumed to belong to the interval $[0,T]$. The constant $K$ will denote a constant which depends only on $\nu$ and
\begin{equation*}
 \sup_{0<\ep\leq1}\|\vep\|_{L^\infty(0,T;L^2)\cap L^2(0,T;H^1)}
\end{equation*}
and whose value may change from one line to another. In particular, the constant $K$ does not depend on $\ep$.

By hypothesis we know that
\begin{equation}\label{prophep}
 \lim_{\ep\to0} \sup_{[0,T]}\ep|h'_\ep(t)|=0.
\end{equation}

We assume that $\ep\leq1/100$ and we choose $\ale$ such that
\begin{equation}
  \label{ale}
100\leq\ale\leq\frac1\ep,\quad \lim_{\ep\to0} \ale=\infty\quad\text{and}\quad \lim_{\ep\to0} \ep\ale(1+|h'_\ep(t)|)=0
\end{equation}
uniformly in $t\in[0,T]$. The existence of such an $\ale$ follows from \eqref{prophep}. Indeed, we could choose for instance
\begin{equation*}
\ale=\max\Bigl(100, \frac1{\sup\limits_{[0,T]}\sqrt{\ep+\ep|h_\ep'(t)|}}\Bigr).  
\end{equation*}

We construct in the following lemma a special cut-off function $\fep$ near the disk $\dep$ such that  $ f _\ep(x)=0$ for all $|x|\leq\ep$ and $ f _\ep(x)=1$ for all $|x|\geq\ep\ale$. 
\begin{lem}\label{flem}
There exists a smooth cut-off function $\fep\in C^\infty(\R^2;[0,1])$ such that
\begin{enumerate}
\item $\fep$ vanishes in the neighborhood of the disk $\overline{D}(0,\ep)$ and $\fep=1$ for $|x|\geq\ep\ale$;
\item there  exists a universal constant $C$ such that
  \begin{equation*}
\nl\infty{ f _\ep}=1,\quad  \nl2{\nabla f _\ep}\leq \frac{C}{\sqrt{\ln\ale}} ,\quad  \nll2{|x|\nabla^2 f _\ep}\leq\frac{C}{\sqrt{\ln\ale}}  \quad\text{and}\quad\nl2{ f _\ep-1}\leq C\frac{\ep\ale}{\ln\ale}\cdot
  \end{equation*}
\end{enumerate}
\end{lem}
\begin{proof}

From Lemma \ref{f_AB} we observe that the function 
\begin{equation*}
  \fept= f _{\ep, \ep\ale}=
\begin{cases}
0&\text{ if }|x|<\ep\\
\frac{\ln (|x|/\ep)}{\ln \ale}&\text{ if }\ep<|x|<\ep\ale\\
1&\text{ if }|x|>\ep\ale
\end{cases}
\end{equation*}
satisfies 
\begin{gather*}
\nl\infty{ \fept}=1,\quad  \nl2{\nabla \fept}\leq \frac{C}{\sqrt{\ln\ale}} ,\quad  \bigl\||x|\nabla^2 \fept\bigr\|_{L^2(\ep<|x|<\ep\ale)}\leq\frac{C}{\sqrt{\ln\ale}} 
\intertext{and} 
\nl2{ \fept -1}\leq C\frac{\ep\ale}{\ln\ale}
  \end{gather*}
so it has all the required properties except smoothness. More precisely, $\fept$ is not smooth across $|x|=\ep$ and $|x|=\ep\ale$. To obtain a smooth function $\fep$ from $\fept$ we need to cut-off in the neighborhood of these two circles. 

Let $g\in C^\infty_0(\R^2;[0,1])$ be such that $g(x)=0$ for $|x|<2$ and $g(x)=1$ for $|x|>4$. We define 
\begin{gather*}
g^1_\ep(x)=g\bigl(\frac x\ep\bigr)=
\begin{cases}
  0,&|x|<2\ep\\
1,&|x|>4\ep
\end{cases}
\intertext{and}
g^2_\ep(x)=1-g\bigl(\frac {8x}{\ep\ale}\bigr)=
\begin{cases}
  1,&|x|<\frac{\ep\ale}4\\
0,&|x|>\frac{\ep\ale}2.
\end{cases}
\end{gather*}
With the help of all the auxiliary functions above, we define a new function 
\begin{gather*}
f_\ep=1+g^2_\ep\bigl(g^1_\ep \fept-1 \bigr)=
\begin{cases}
  1,&|x|>\frac{\ep\ale}2\\
  1+g^2_\ep\bigl(\fept-1 \bigr),&\frac{\ep\ale}4<|x|<\frac{\ep\ale}2 \\
  \fept,&4\ep<|x|<\frac{\ep\ale}4\\
  g^1_\ep \fept,&2\ep<|x|<4\ep\\
  0,&|x|<2\ep.
\end{cases}
\end{gather*}
Clearly $f_\ep$ satisfies (i) and is smooth across $|x|=\ep$ and $|x|= \ep \ale$, so it remains to prove (ii). From the definition of $f_\ep$, we immediately see that $\nl\infty{ f _\ep}=1$. To simplify the write-up, we use the notation $L^p(a,b)=L^p(a<|x|<b)$. Clearly $g_\ep^1$ and $g_\ep^2$ are uniformly bounded in $L^\infty$ and $\nabla g_\ep^1$ and $\nabla g^2_\ep$ are uniformly bounded in $L^2$. Using these observations we estimate
\begin{align*}
\norm{\nabla f_\ep }_{L^{2}}
& \leq  \norm{\na \big(g^1_\ep \fept \big)}_{L^{2}(2\ep,4\ep)}
 + \norm{\na \fept}_{L^{2}(4\ep,\frac{\ep\ale}4)} + \norm{\nabla\big( g^2_\ep\bigl( \fept-1 \bigr) \big)}_{L^{2}(\frac{\ep\ale}4,\frac{\ep\ale}2)}\\
& \leq C\bigl(\norm{\na \fept}_{L^{2}} + \norm{\fept}_{L^{\infty}(2\ep,4\ep)} + \norm{\fept-1}_{L^{\infty}(\frac{\ep\ale}4,\frac{\ep\ale}2)} \bigr)\\
& \leq \frac{C}{\sqrt{\ln\ale}} + \frac{C}{\ln\ale}\\
& \leq \frac{C}{\sqrt{\ln\ale}}
\end{align*}
where we used the bounds
\begin{gather*}
\norm{\fept}_{L^{\infty}(2\ep,4\ep)} =\normm{\frac{\ln (|x|/\ep) }{\ln \ale}}_{L^{\infty}(2\ep,4\ep)} \leq \frac{C}{\ln\ale} 
\intertext{and}  
\norm{\fept-1}_{L^{\infty}(\frac{\ep\ale}4,\frac{\ep\ale}2)} =
\normm{\frac{\ln (|x|/ (\ep \ale)) }{\ln \ale} }_{L^{\infty}(\frac{\ep\ale}4,\frac{\ep\ale}2)} \leq \frac{C}{\ln\ale}.
\end{gather*}

Similarly, using in addition that $\nll\infty{|x|\nabla g^i_\ep}$ and $\nll2{|x|\nabla^2 g^i_\ep}$ are bounded independently of $\ep$ for $i=1,2$, we can estimate 
\begin{align*}
\normm{|x| \nabla^2 f_\ep }_{L^{2}}
& \leq \normm{|x| \na^2 \big(g^1_\ep \fept \big)}_{L^{2}(2\ep,4\ep)}
 + \normm{|x| \na^2 \fept}_{L^{2}(4\ep,\frac{\ep\ale}4)} \\
&\hskip 5cm+ \normm{|x|\nabla^2 \big( g^2_\ep\bigl( \fept-1 \bigr) \big)}_{L^{2}(\frac{\ep\ale}4,\frac{\ep\ale}2)}\\
& \leq C\bigl( \normm{|x| \na^2 \fept}_{L^{2}} + \norm{\na \fept}_{L^{2}} + \norm{\fept}_{L^{\infty}(2\ep,4\ep)} + \norm{\fept -1 }_{L^{\infty}(\frac{\ep\ale}4,\frac{\ep\ale}2)}\bigr)\\
& \leq \frac{C}{\sqrt{\ln\ale}} + \frac{C}{\ln\ale} \\& \leq \frac{C}{\sqrt{\ln\ale}}.
\end{align*}

Finally, 
\begin{align*}
\norm{f_\ep-1}_{L^{2}} 
& \leq  \norm{g^1_\ep \fept-1 }_{L^{2}(2\ep,4\ep)}+  \norm{ \fept-1 }_{L^{2}(4\ep,\frac{\ep\ale}4)}+  \norm{ g^2_\ep\bigl(\fept-1 \bigr)}_{L^{2}(\frac{\ep\ale}4,\frac{\ep\ale}2)} +\|1\|_{L^2(|x|<2\ep)}\\
& \leq   C \bigl(\norm{ \fept-1 }_{L^{2}} +\|1\|_{L^2(|x|<4\ep)}\bigr)\\
& \leq C\bigl( \frac{\ep \ale}{\ln \ale} + \ep \bigr) \\
&\leq C\frac{\ep \ale}{\ln \ale}.
\end{align*}
This completes the proof of the lemma.
\end{proof}

The function $\fep$ is a cut-off in the neighborhood of the disk $D(0,\ep)$. We define now  a cut-off in the neighborhood of the disk $D(h_\ep(t),\ep)$ by setting
\begin{equation*}
\etaep(t,x)=\fep(x-h_\ep(t)).  
\end{equation*}
Lemma \ref{flem} immediately implies that $\etaep$ has the following properties:
\begin{lem}\label{etaeplem}
We have that
\begin{enumerate}
\item $\etaep\in \Lip (\R_+;C^\infty_{0,\sigma})$;
\item $\etaep$ vanishes in the neighborhood of the disk $\depf$ and $\etaep=1$ for $|x-h_\ep(t)|\geq\ep\ale$;
\item there  exists a universal constant $C$ such that
  \begin{gather}
\nl\infty{ \etaep}=1,\quad  \nl2{\nabla \etaep}\leq \frac{C}{\sqrt{\ln\ale}} ,\quad  \nll2{|x-h_\ep(t)|\nabla^2\etaep}\leq\frac{C}{\sqrt{\ln\ale}}  \label{etap1}
\intertext{and} 
\nl2{\etaep-1}\leq C\frac{\ep\ale}{\ln\ale}\cdot\label{etap2}
  \end{gather}
\end{enumerate} 
\end{lem}

Given a test function  $\varphi\in C^1_b (\R_+;C^\infty_{0,\sigma})$ we construct  a test function  $\varphi_\ep$ on the set $|x-h_\ep(t)|>\ep$ by setting
\begin{equation}\label{cutof}
\varphi_{\ep}= \na^{\perp}(\eta_{\ep} \psiep)
\end{equation}
where $\psiep$ was defined in Section \ref{sectnot} (see relation \eqref{psiepdef}). 
We state some properties of $\varphi_\ep$ in the following lemma:
\begin{lem}{\label{lemmaphi}}
The test function $\varphi_\ep$ has the following properties:
\begin{enumerate}
\item $\varphi_{\ep}\in \Lip (\R_+;C^\infty_{0,\sigma})$ and is supported in the set $|x-h_\ep(t)|>\ep$;
\item $\varphi_{\ep} \to \varphi$ strongly in $L^\infty(0,T;H^1)$ as $\ep\to0$;
\item there exists a universal constant $C$ such that 
  \begin{equation}
    \label{h1phi}
\|\varphi_{\ep}\|_{L^\infty(0,T;H^{1})} \leq C \|\varphi\|_{L^\infty(0,T;H^3)}.    
  \end{equation}
\end{enumerate}
\end{lem}
\begin{proof}
Since $\etaep$ and $\psiep$ are  $\Lip $ in time and smooth in space, so is $\varphi_\ep$. The compact support in $x$ of $\varphi_\ep$ in the set $|x-h_\ep(t)|>\ep$ follows from the compact support of $\varphi$ and the localization properties of $\etaep$. Obviously $\phiep$ is also divergence free so claim (i) follows.

Recalling that $\nabla^\perp\psiep=\varphi$  we write
\begin{equation*}
 \varphi_{\ep} - \varphi=\nabla^\perp(\etaep\psiep)-\varphi
= \nabla^\perp \etaep\psiep +\etaep \nabla^\perp\psiep-\varphi
=\nabla^\perp \etaep\psiep+(\etaep-1)\varphi.
\end{equation*}

Using the bound \eqref{psiep2} and recalling that $\nabla\etaep$ is supported in $D(h_\ep(t),\ep\ale)$ we can estimate
\begin{align*}
\norm{\varphi_{\ep} - \varphi}_{L^{2}} 
& \leq  \norm{(\eta_{\ep} -1) \varphi}_{L^{2}}+ \norm{\na \eta_{\ep} \psiep  }_{L^{2}}\\
& \leq  \norm{\eta_{\ep}-1}_{L^{2}} \norm{\varphi}_{L^{\infty}} + \norm{\na \eta_{\ep}}_{L^{2}} \norm{\psiep }_{L^{\infty}(D(h_\ep(t),\ep\ale))}\\
& \leq  \norm{\eta_{\ep}-1}_{L^{2}} \norm{\varphi}_{L^{\infty}} + \ep\ale\norm{\na \eta_{\ep}}_{L^{2}} \nl\infty\varphi\\
&= \nl\infty\varphi(\norm{\eta_{\ep}-1}_{L^{2}} +\ep\ale\norm{\na \eta_{\ep}}_{L^{2}} ).
\end{align*}
Taking the supremum on $[0,T]$ and using \eqref{ale}, \eqref{etap1} and \eqref{etap2} we deduce that
\begin{equation}\label{convl2}
\norm{\varphi_{\ep} - \varphi}_{L^\infty(0,T;L^{2})}\leq C\frac{\ep\ale}{\sqrt{\ln\ale}}\norm{\varphi}_{L^\infty([0,T]\times\R^2)}  \toep0. 
\end{equation}

Next,
\begin{align*}
\nl2 { \na (\varphi_{\ep} - \varphi)}
&= \nl2 { \na \nabla^\perp\bigl((\eta_\ep-1)\psi_\ep\bigr)}\\
&\leq \nl2{ \na \nabla^\perp\eta_\ep\psiep}+ C\nl2{\nabla\eta_\ep}\nl\infty{\nabla\psi_\ep}+\nl2{\eta_\ep-1}\nl\infty{\nabla^2\psi_\ep}.
\end{align*}
We bound the first term on the right-hand side using \eqref{psiep1} and \eqref{etap1}:
\begin{equation*}
\nl2{ \na \nabla^\perp\eta_\ep\psiep}\leq C\nl\infty\varphi\nll2{|x-h_\ep(t)|\nabla^2\etaep} 
\leq \frac{C}{\sqrt{\ln\ale}}\nl\infty\varphi.
\end{equation*}
Recalling that $\nabla^\perp\psiep=\varphi$ and using again Lemma \ref{etaeplem} we infer that
\begin{align*}
\nl2 { \na (\varphi_{\ep} - \varphi)}
&\leq \frac{C}{\sqrt{\ln\ale}}\nl\infty\varphi  + C\nl2{\nabla\eta_\ep}\nl\infty{\varphi}+ \nl2{\eta_\ep-1}\nl\infty{\nabla\varphi}\\
&\leq \frac{C}{\sqrt{\ln\ale}}\lip\varphi.
\end{align*}
Combining this bound with \eqref{convl2} implies that 
\begin{equation*}
\norm{\varphi_{\ep} - \varphi}_{L^\infty(0,T;H^1)}\leq \frac{C}{\sqrt{\ln\ale}}\norm{\varphi}_{L^\infty(0,T;W^{1,\infty})}  \toep0. 
\end{equation*}
In addition, we obtain that there exists a universal constant  $C>0$ such that
\begin{equation*}
\norm{\varphi_{\ep}}_{L^\infty(0,T;H^1)} 
\leq C\norm{\varphi}_{L^\infty(0,T;H^1\cap W^{1,\infty})}.
\end{equation*}

Using the Sobolev embedding $H^3\hookrightarrow W^{1,\infty}$ completes the proof of the lemma.
\end{proof}

We end this section with an estimate on the $H^{-1}$ norm of the time-derivative of $\phiep$.
\begin{lem}\label{timeder}
Let $w$ be an $H^1$ vector field. There exists a universal constant $C>0$ such that for all times $t\geq0$ where $h_\ep$ is differentiable we have that
\begin{multline*}
  \bigl|\int_{\R^2}w(x)\cdot\bigl(\pat\phiep(t,x)-\pat\varphi(t,x)\bigr)dx\bigr|\leq 
C\nl2{\curl w}\bigl(\frac{\ep^2\ale^2}{\ln\ale}\nl\infty{\pat\varphi(t,\cdot)}\\
+\frac{\ep\ale}{\sqrt{\ln\ale}}|h'_\ep(t)|\nl\infty{\varphi(t,\cdot)} \bigr).
\end{multline*}
\end{lem}
\begin{proof}
Let $t$ be a time where $h_\ep$ is differentiable. 
We use \eqref{cutof} to write
\begin{align*}
 \int_{\R^2}w(x)\cdot\bigl(\pat\phiep(t,x)-\pat\varphi(t,x)\bigr)dx
&=    \int_{\RR^{2}}  w  \cdot \pat \nabla^\perp\bigl((\eta_\ep-1)\psi_{\ep }\bigr)\\
&=   
-  \int_{\RR^{2}} \curl w \,  \pat \bigl((\eta_\ep-1)\psi_{\ep }\bigr)\\
&=  -  \int_{\RR^{2}} \curl w\,  \pat \eta_\ep\psi_{\ep }-  \int_{\RR^{2}} \curl w\,   (\eta_\ep-1)\pat\psi_{\ep }.
\end{align*}

Clearly
\begin{equation*}
  \pat\eta_\ep=\pat(f_\ep(x-h_\ep(t)))=-h'_\ep(t)\cdot\nabla f_\ep(x-h_\ep(t))
\end{equation*}
is supported in the set $\{|x-h_\ep(t)|\leq \ep\ale\}$. 
We can therefore bound
\begin{align*}
  \Bigl| \int_{\RR^{2}} \curl w\, \pat \eta_\ep\psi_{\ep }\Bigr|
&\leq C|h_\ep'(t)| \int_{|x-h_\ep(t)|\leq \ep\ale} |\curl w||\nabla f_\ep(x-h_\ep(t))||\psiep|\\
&\leq C |h_\ep'(t)|\nl2{\curl w}\nl2{\nabla f_\ep}\norm{\psiep}_{L^{\infty}(D(h_\ep(t),\ep\ale))}\\
&\leq C\frac{\ep\ale}{\sqrt{\ln\ale}} |h_\ep'(t)|\nl2{\curl w}\nl\infty\varphi
\end{align*}
where we used \eqref{psiep2} and Lemma \ref{flem}.

Similarly,  $\eta_\ep-1$  is supported in the set $\{|x-h_\ep(t)|\leq \ep\ale\}$ so we can use \eqref{psiep3} and \eqref{etap2} to deduce that
\begin{align*}
  \bigl| \int_{\RR^{2}} \curl w\,  (\eta_\ep-1)\pat\psi_{\ep }\bigr|
&\leq \nl2{\curl w}\nl2{\etaep-1} \norm{\pat\psiep}_{L^{\infty}(D(h_\ep(t),\ep\ale))}   \\
&\leq C\frac{\ep\ale}{\ln\ale}\nl2{\curl w}(\ep\ale\nl\infty{\pat\varphi}+|h'_\ep(t)|\nl\infty{\varphi}).
\end{align*}
The conclusion follows putting together the above relations.
\end{proof}

\section{Temporal estimate and strong convergence}
\label{temp}

The aim of this section is to prove the strong convergence of some sub-sequence of $\vep$. More precisely, we will prove the following result.
\begin{lem}\label{lemstrong}
There exists  a sub-sequence $\vepk$ of $\vep$ which converges strongly in $L^2(0,T;L^2\loc)$. 
\end{lem}

To prove this lemma  we first show some time-derivative estimates and then use the Ascoli theorem.

Let $\varphi \in C^\infty_{0,\sigma}(\RR^{2})$ be a test function which does not depend on the time. Even though $\varphi$ does not depend on $t$, we can still perform the construction of the cut-off $\varphi_{\ep}$ as in Section \ref{defcutof} (see relation \eqref{cutof}) and all the results of that section remain valid. Observe that even though $\varphi$ does not depend on the time, the modified test function $\varphi_\ep$ is time-dependent. 

Let us denote by $H^s_\sigma$ the space of $H^s$ divergence free vector fields on $\R^2$. We endow  $H^s_\sigma$ with the $H^s$ norm. The dual space of $H^s_\sigma$ is $H^{-s}_\sigma$.  We have that $C^\infty_{0,\sigma}$ is dense in $H^s_\sigma$ for all $s\in\R$.

Let $t\in[0,T]$ be fixed. We use Lemma \ref{etaeplem} and relation \eqref{psiep2} to bound
\begin{align*}
\Bigl|\int_{\R^2}\vep(t,x)\cdot\varphi_\ep(t,x)\,dx\Bigr|
&= \Bigl|\int_{\R^2}\vep\cdot\nabla^\perp(\eta_\ep\psi_\ep)\,dx\Bigr|\\
& =\Bigl|\int_{\R^2}\vep\cdot(\nabla^\perp\eta_\ep\psi_\ep+\eta_\ep\varphi)\,dx\Bigr|\\
&\leq \nl2\vep\nl2{\nabla\eta_\ep} \norm{\psiep }_{L^{\infty}(D(h_\ep(t),\ep\ale))}+\nl2\vep\nl\infty{\eta_\ep}\nl2\varphi\\
&\leq C\frac{\ep\ale}{\sqrt{\ln\ale}}\nl2\vep\nl\infty\varphi   +\nl2\vep\nl2\varphi\\
&\leq K_1\nh2\varphi
\end{align*}
for some constant $K_1$ independent of $\ep$ and $t$. We used above the Sobolev embedding $H^2\hookrightarrow L^\infty$, the boundedness of $\vep$ in $L^\infty(0,T;L^2)$ and relations \eqref{ale} and \eqref{etap1}.
We infer that, for fixed $t$, the map
\begin{equation*}
C^\infty_{0,\sigma}\ni\varphi\mapsto \int_{\R^2}\vep(t,x)\cdot\varphi_\ep(t,x)\,dx\in\R  
\end{equation*}
is linear and continuous for the $H^2$ norm. Since the closure of $C^\infty_{0,\sigma}$ for the $H^2$ norm is $H^2_\sigma$, the above map can be uniquely extended to a continuous linear mapping from $H^2_\sigma$ to $\R$. Therefore it can be identified to an element of the dual of  $H^2_\sigma$ which is $H^{-2}_\sigma$. We conclude that there exists some $\xiep(t)\in H^{-2}_\sigma$ such that
\begin{equation*}
\langle \xiep(t),\varphi\rangle= \int_{\R^2}\vep(t,x)\cdot\varphi_\ep(t,x)\,dx\qquad\forall\varphi \in H^2_\sigma. 
\end{equation*}
Above $\langle\cdot,\cdot\rangle$ denotes the duality bracket between $H^{-2}_\sigma$ and $H^2_\sigma$ which is the extension of the usual $L^2$ scalar product. In addition, we have that $\nh{-2}{\xiep(t)}\leq K_1$, so $\xiep$ belongs to the space $L^\infty(0,T;H^{-2}_\sigma)$ and is bounded independently of $\ep$ in this space.

Because $\varphi_\ep$ is compactly supported in $\{|x-h_\ep(t)|>\ep\}$ it can be used as test function in \eqref{nsep}. Multiplying \eqref{nsep} by $\varphi_\ep$ and integrating in space and time from $s$ to $t$ yields
\begin{equation*}
\int_s^t\int_{\R^2}\partial_\tau\vep\cdot\varphi_\ep+ \nu\int_s^t\int_{\R^2}\nabla\vep:\nabla\varphi_\ep+ \int_s^t\int_{\R^2}\vep\cdot\nabla\vep\cdot\varphi_\ep=0.
\end{equation*}

We integrate by parts in time the first term above:
\begin{align*}
\int_s^t\int_{\R^2}\partial_\tau\vep\cdot\varphi_\ep
&=\int_{\R^2}  \vep(t,x)\cdot\varphi_\ep(t,x)\,dx-\int_{\R^2}  \vep(s,x)\cdot\varphi_\ep(s,x)\,dx -\int_s^t\int_{\R^2}\vep\cdot\partial_\tau\varphi_\ep\\
&=\langle \xiep(t)-\xiep(s),\varphi\rangle-\int_s^t\int_{\R^2}\vep\cdot\partial_\tau\varphi_\ep.
\end{align*}

We deduce that
\begin{equation}\label{xidif}
\langle \xiep(t)-\xiep(s),\varphi\rangle=  \int_s^t\int_{\R^2}\vep\cdot\partial_\tau\varphi_\ep
-\nu\int_s^t\int_{\R^2}\nabla\vep:\nabla\varphi_\ep
-\int_s^t\int_{\R^2}\vep\cdot\nabla\vep\cdot\varphi_\ep.
\end{equation}

We bound first
\begin{align*}
 \bigl|\nu \int_{s}^{t} \int_{\RR^{2}} \na \vep :\na \varphi_{\ep}\bigr|
&\leq \nu \int_{s}^{t} \norm{\na \vep}_{L^{2}} \norm{\na \varphi_{\ep}}_{L^{2}} \\ 
& \leq C\nu(t-s)^{\frac{1}{2}} \norm{\varphi}_{H^{3}}  \norm{\na \vep}_{L^{2}([0,T]\times\RR^{2})} \\ 
& \leq K(t-s)^{\frac{1}{2}}\nh3\varphi
\end{align*}
where we used \eqref{h1phi} and the hypothesis that $\vep$ is bounded in $L^2(0,T;H^1)$.

To estimate the last term in \eqref{xidif} we use the Gagliardo-Nirenberg inequality $\nl4f\leq C\nl2f^{\frac12}\nl2{\nabla f}^{\frac12}$, the boundedness of $\vep$ in the space displayed in \eqref{vepspace} and relation \eqref{h1phi}:
\begin{align*}
\bigl|\int_{s}^{t} \int_{\RR^{2}} \vep \cdot \na \vep \cdot \varphi_{\ep} \bigl|
&=|\int_{s}^{t} \int_{\RR^{2}} \vep \cdot \na \phiep \cdot \vep \bigl|\\
&\leq \int_{s}^{t} \nl4\vep^2\nl2{\nabla\phiep} \\
&\leq\int_s^t \nl2\vep\nl2{\nabla\vep}\nh1\phiep\\
&\leq C(t-s)^{\frac12}\|\vep\|_{L^\infty(0,T;L^2)}\|\nabla \vep\|_{L^2([0,T]\times\R^2)}\nh3\varphi\\
&\leq K(t-s)^{\frac12}\nh3\varphi.
\end{align*}

It remains to estimate the first term on the right-hand side of \eqref{xidif}. To do that, we use Lemma \ref{timeder}. Recalling that $\varphi$ does not depend on the time, we can write
\begin{align*}
\bigl|\int_s^t\int_{\R^2}\vep\cdot\partial_\tau\varphi_\ep\bigr|
&\leq C\frac{\ep\ale}{\sqrt{\ln\ale}}\int_s^t \nl2{\omep}
|h'_\ep|\nl\infty{\varphi} \\
&\leq C\frac{\ep\ale}{\sqrt{\ln\ale}}\nh2\varphi\int_s^t \nl2\omep|h'_\ep|\\
&\leq C(t-s)^{\frac12}\nh2\varphi\frac{\ep\ale}{\sqrt{\ln\ale}}\sup_{[0,T]}|h'_\ep|\|\omep\|_{L^2([0,T]\times\R^2)}.
\end{align*}

Due to the hypothesis imposed on $\ale$, see \eqref{ale}, we know that $\frac{\ep\ale}{\sqrt{\ln\ale}}\sup_{[0,T]}|h'_\ep|$ goes to 0 as $\ep\to0$. In particular it is bounded uniformly in $\ep$. 

Recalling again the boundedness of $\vep$ in the space $L^2(0,T;H^1)$, we infer from the above relations that 
\begin{equation*}
\bigl|\langle \xiep(t)-\xiep(s),\varphi\rangle\bigr|
\leq K(t-s)^{\frac{1}{2}} \nh3\varphi  
\end{equation*}
where the constant $K$ does not depend on $\ep$ and $\varphi$.
By density of  $C^\infty_{0,\sigma}$ in $H^3_\sigma$ we infer that $\nh{-3}{\xiep(t)-\xiep(s)}\leq K(t-s)^{\frac{1}{2}}$. The functions $\xiep(t)$ are therefore equicontinuous in time with values in $H^{-3}_\sigma$. They are also bounded in  $H^{-3}_\sigma$ because we already know that they are bounded in  $H^{-2}_\sigma$. Since the embedding $H^{-3}\hookrightarrow H^{-4}\loc$ is compact, the Ascoli theorem implies that there exists a subsequence $\xiepk$  of $\xiep$ which converges strongly in $C^0([0,T];H^{-4}\loc)$.

Recalling the definition of $\xiep$ and using Lemma \ref{etaeplem} we can write 
\begin{align*}
|\langle \xiep(t)-\vep(t),\varphi\rangle| 
&=\bigl|\int_{\R^2}\vep\cdot(\nabla^\perp\eta_\ep\psi_\ep+\eta_\ep\varphi)\,dx-\int_{\R^2}v_\ep\cdot\varphi \bigr|\\  
&=\bigl|\int_{\R^2}\vep\cdot(\nabla^\perp\eta_\ep\psi_\ep+(\eta_\ep-1)\varphi)\,dx\bigr|\\
&\leq C\nl2\vep\nl2{\nabla\eta_\ep}\norm{\psiep }_{L^{\infty}(D(h_\ep(t),\ep\ale))}+C\nl2\vep\nl2{\eta_\ep-1}\nl\infty\varphi\\
&\leq C\nl2\vep\nl\infty\varphi\bigl(\frac{\ep\ale}{\sqrt{\ln\ale}}+\frac{\ep\ale}{\ln\ale}\bigr)\\
&\leq  C\nl2\vep\nh2\varphi \frac{\ep\ale}{\sqrt{\ln\ale}}.
\end{align*}

Hence
\begin{equation*}
  \nh{-2}{\xiep(t)-\vep(t)}\leq C\nl2\vep\frac{\ep\ale}{\sqrt{\ln\ale}}\toep0
\end{equation*}
uniformly in time. Recalling that $\xiepk$ converges strongly in $H^{-4}\loc$ uniformly in time we infer that $\vepk$ also converges strongly in $L^\infty(0,T;H^{-4}\loc)$. The interpolation inequality $\nl2\cdot\leq\nh{-4}\cdot^{\frac1{5}}\nh1{\cdot}^{\frac4{5}}$ and the boundedness of $\vep$ in $L^2(0,T;H^1)$ finally imply that $\vepk$ converges strongly in $L^{\frac52}(0,T;L^2\loc)\hookrightarrow L^2(0,T;L^2\loc)$. This completes the proof of Lemma \ref{lemstrong}.

\section{Passing to the limit}
\label{paslim}

In this section we complete the proof of Theorem \ref{mainthm}. It is now only a matter of putting together the results proved in the previous sections.

Given the boundedness of $\vep$ in $L^\infty(0,T;L^2)\cap L^2(0,T;H^1)$ and Lemma \ref{lemstrong}, we know that there exists some $v\in L^\infty(0,T;L^2)\cap L^2(0,T;H^1)$ and some sub-sequence $\vepk$ such that
\begin{gather}
\vepk\rightharpoonup v \quad\text{weak$\ast$ in } L^\infty(0,T;L^2)\notag\\
\vepk\rightharpoonup v \quad\text{weakly in } L^2(0,T;H^1)\label{27}\\
\intertext{and}
\vepk\to v \quad\text{strongly in } L^2(0,T;L^2\loc). \label{28}
\end{gather}

Let $\varphi \in C_{0}^{\infty}([0,T) \times \RR^{2})$ be a divergence-free vector field. We construct $\varphi_{{\ep_k} }$ as in Section \ref{defcutof}, see relation \eqref{cutof}. Since $\varphi_{\ep_k}$ is compactly supported in the set $\{|x-h_{\ep_k}(t)|>{\ep_k}\}$, we can use it as test function in \eqref{nsep} written  for $\ep_k$. We multiply \eqref{nsep} by $\varphi_{\ep_k}$ and integrate by parts in time and space to obtain that
\begin{multline}\label{equality}
-\int_{0}^{T} \int_{\RR^{2}} \vepk \cdot \pat \varphi_{{\ep_k} } + \nu \int_{0}^{T} \int_{\RR^{2}} \na \vepk : \na \varphi_{{\ep_k} } + \int_{0}^{T} \int_{\RR^{2}} \vepk \cdot \na \vepk \cdot \varphi_{{\ep_k} } \\
= \int_{\RR^{2}} \vepk(0)  \cdot\varphi_{{\ep_k} }(0). 
\end{multline}

We will pass to the limit ${\ep_k}\to0$ in each of the terms above.

First, we know by hypothesis that $\vepk(0)\rightharpoonup v_0$ weakly in $L^2$. From Lemma \ref{lemmaphi} we also have that $\varphi_{{\ep_k} } (0) \rightarrow \varphi(0)$ strongly in $L^{2}$, so 
\begin{equation}{\label{converge3}}
\int_{\RR^{2}} v_{{\ep_k} }(0) \cdot \varphi_{{\ep_k} }(0)\toepk \int_{\RR^{2}} v(0)\cdot \varphi(0).
\end{equation}

Next, we also know from Lemma \ref{lemmaphi} that $\na \varphi_{{\ep_k} } \rightarrow \na \varphi$ strongly in $L^{2}([0,T] \times \RR^{2})$. Given that $\nabla\vepk\rightharpoonup\nabla v$ weakly in  $L^{2}([0,T]\times \RR^{2})$, see relation \eqref{27}, we infer that
\begin{equation}{\label{converge2}}
\int_{0}^{T}  \int_{\RR^{2}} \na \vepk : \na \varphi_{{\ep_k} } \toepk \int_{0}^{T}  \int_{\RR^{2}} \na v : \na \varphi.
\end{equation}

The nonlinear term also passes to the limit quite easily. We decompose
\begin{equation*}
\int_{0}^{T} \int_{\RR^{2}} \vepk \cdot \na \vepk \cdot \phiepk
= \int_{0}^{T} \int_{\RR^{2}} \vepk \cdot \na \vepk \cdot \varphi
+\int_{0}^{T} \int_{\RR^{2}} \vepk \cdot \na \vepk \cdot (\phiepk-\varphi).
\end{equation*}

We know from \eqref{28} that $\vepk\to v$ strongly in $L^2(0,T;L^2\loc)$, from \eqref{27} that $\nabla\vepk\rightharpoonup\nabla v$ weakly in $L^2(0,T;L^2)$. Recalling that $\varphi$ is compactly supported and since we obviously have that $\varphi$ is uniformly bounded in space and time we can pass to the limit in the first term on the right-hand side:
\begin{equation*}
\int_{0}^{T} \int_{\RR^{2}} \vepk \cdot \na \vepk \cdot \varphi
\toepk\int_{0}^{T} \int_{\RR^{2}} v \cdot \na v \cdot \varphi.   
\end{equation*}

To pass to the limit in the second term we  make an integration by parts and use the Hölder inequality, the Gagliardo-Nirenberg inequality $\nl4f\leq C\nl2f^{\frac12}\nl2{\nabla f}^{\frac12}$ and Lemma \ref{lemmaphi}
\begin{align*}
\bigl|\int_{0}^{T} \int_{\RR^{2}} \vepk \cdot \na \vepk \cdot (\phiepk-\varphi) \bigr|
&= \bigl|\int_{0}^{T} \int_{\RR^{2}} \vepk \otimes \vepk : \nabla(\phiepk-\varphi) \bigr| \\
&\leq\int_0^T\nl4\vepk^2\nl2{\nabla(\phiepk-\varphi)}\\
&\leq C\int_0^T\nl2\vepk\nl2{\nabla\vepk}\nl2{\nabla(\phiepk-\varphi)}\\
&\leq CT^{\frac12}\|\vepk\|_{L^\infty(0,T;L^2)}\|\vepk\|_{L^2(0,T;H^1)}\|\phiepk-\varphi\|_{L^\infty(0,T;H^1)}\\
&\toepk0.
\end{align*}
 We infer that
 \begin{equation}\label{converge1}
\int_{0}^{T} \int_{\RR^{2}} \vepk \cdot \na \vepk \cdot \phiepk
\toepk\int_{0}^{T} \int_{\RR^{2}} v \cdot \na v \cdot \varphi.    
 \end{equation}

The last term we need to pass to the limit is the term with the time-derivative. Thanks to Lemma \ref{timeder} we can bound
\begin{align*}
  \bigl|\int_0^T\int_{\R^2}\vepk\cdot\bigl(\pat\phiepk-\pat\varphi\bigr)dx\bigr|
&\leq C\int_0^T\nl2{\curl \vepk}\bigl(\frac{{\ep_k}^2\alek^2}{\ln\alek}\nl\infty{\pat\varphi}+\frac{{\ep_k}\alek}{\sqrt{\ln\alek}}|h'_{\ep_k}(t)|\nl\infty{\varphi} \bigr)\\
&\hskip-3cm \leq CT^{\frac12}\|\vepk\|_{L^2(0,T;H^1)}\|\varphi\|_{W^{1,\infty}(0,T;L^\infty)}\max\bigl(\frac{{\ep_k}^2\alek^2}{\ln\alek},\frac{{\ep_k}\alek}{\sqrt{\ln\alek}} |h_{\ep_k}'(t)|\bigr)\\
&\hskip-3cm  \toepk0
\end{align*}
where we used \eqref{ale}. But we also have that
\begin{equation*}
\int_{0}^{T} \int_{\RR^{2}} \vepk \cdot \pat \varphi\toepk  \int_{0}^{T} \int_{\RR^{2}} v \cdot \pat \varphi 
\end{equation*}
so we can conclude that
\begin{equation}{\label{converge4}}
\int_{0}^{T} \int_{\RR^{2}} \vepk \pat \varphi_{{\ep_k} } \toepk \int_{0}^{T} \int_{\RR^{2}} v \cdot \pat \varphi .
\end{equation}
Gathering \eqref{equality}, \eqref{converge3}, \eqref{converge2}, \eqref{converge1} and \eqref{converge4}, we conclude that
\begin{equation*}
-\int_{0}^{T} \int_{\RR^{2}} v \cdot \partial_{t}\varphi + \nu \int_{0}^{T} \int_{\RR^{2}} \nabla v :\nabla \varphi + \int_{0}^{T} \int_{\RR^{2}} v\cdot \nabla v\cdot \varphi = \int_{\RR^{2}} v(0) \cdot\varphi(0)
\end{equation*}
which is the weak formulation of Navier-Stokes equations in $\RR^{2}$. This completes the proof of Proposition \ref{prop}.

\section*{Acknowledgments}   J.H. and D.I. have been partially funded by the ANR project Dyficolti ANR-13-BS01-0003-01. D.I. has been  partially funded by the LABEX MILYON (ANR-10-LABX-0070) of Universit\'e de Lyon, within the program ``Investissements d'Avenir'' (ANR-11-IDEX-0007) operated by the French National Research Agency (ANR).

\bigskip

\begin{description}
\item[J. He] Universit\'e de Lyon, Universit\'e Lyon 1 --
CNRS UMR 5208 Institut Camille Jordan --
43 bd. du 11 Novembre 1918 --
Villeurbanne Cedex F-69622, France.\\
Email: \texttt{jiao.he@math.univ-lyon1.fr}
\item[D. Iftimie] Universit\'e de Lyon, Universit\'e Lyon 1 --
CNRS UMR 5208 Institut Camille Jordan --
43 bd. du 11 Novembre 1918 --
Villeurbanne Cedex F-69622, France.\\
Email: \texttt{iftimie@math.univ-lyon1.fr}
\end{description}


\begin{thebibliography}{10}

\bibitem{chipot_limits_2014}
M.~Chipot, G.~Planas, J.~C. Robinson, and W.~Xue.
\newblock Limits of the {Stokes} and {Navier}-{Stokes} equations in a punctured
  periodic domain.
\newblock {\em arXiv:1407.6942 [math]}, 2014.

\bibitem{desjardins_existence_1999}
B.~Desjardins and M.~J. Esteban.
\newblock Existence of weak solutions for the motion of rigid bodies in a
  viscous fluid.
\newblock {\em Archive for Rational Mechanics and Analysis}, 146(1):59--71,
  1999.

\bibitem{ervedoza_long-time_2014}
S.~Ervedoza, M.~Hillairet, and C.~Lacave.
\newblock Long-{Time} {Behavior} for the {Two}-{Dimensional} {Motion} of a
  {Disk} in a {Viscous} {Fluid}.
\newblock {\em Communications in Mathematical Physics}, 329(1):325--382, 2014.

\bibitem{glass_motion_2014-1}
O.~Glass, C.~Lacave, and F.~Sueur.
\newblock On the motion of a small body immersed in a two dimensional
  incompressible perfect fluid.
\newblock {\em Bulletin de la Société Mathématique de France},
  142(3):489--536, 2014.

\bibitem{glass_motion_2016}
O.~Glass, C.~Lacave, and F.~Sueur.
\newblock On the {Motion} of a {Small} {Light} {Body} {Immersed} in a {Two}
  {Dimensional} {Incompressible} {Perfect} {Fluid} with {Vorticity}.
\newblock {\em Communications in Mathematical Physics}, 341(3):1015--1065,
  2016.

\bibitem{glass_point_2018}
O.~Glass, A.~Munnier, and F.~Sueur.
\newblock Point vortex dynamics as zero-radius limit of the motion of a rigid
  body in an irrotational fluid.
\newblock {\em Inventiones mathematicae}, 214(1):171--287, 2018.

\bibitem{hoffmann_motion_1999}
K.-H. Hoffmann and V.~N. Starovoitov.
\newblock On a motion of a solid body in a viscous fluid. {Two}-dimensional
  case.
\newblock {\em Advances in Mathematical Sciences and Applications},
  9(2):633--648, 1999.

\bibitem{iftimie_two_2003}
D.~Iftimie, M.~C. Lopes~Filho, and H.~J. Nussenzveig~Lopes.
\newblock Two {Dimensional} {Incompressible} {Ideal} {Flow} {Around} a {Small}
  {Obstacle}.
\newblock {\em Communications in Partial Differential Equations},
  28(1-2):349--379, 2003.

\bibitem{iftimie_two-dimensional_2006}
D.~Iftimie, M.~C. Lopes~Filho, and H.~J. Nussenzveig~Lopes.
\newblock Two-dimensional incompressible viscous flow around a small obstacle.
\newblock {\em Mathematische Annalen}, 336(2):449--489, 2006.

\bibitem{lacave_two-dimensional_2009}
C.~Lacave.
\newblock Two-dimensional incompressible viscous flow around a thin obstacle
  tending to a curve.
\newblock {\em Proceedings of the Royal Society of Edinburgh Section A:
  Mathematics}, 139(6):1237--1254, 2009.

\bibitem{lacave_small_2017}
C.~Lacave and T.~Takahashi.
\newblock Small {Moving} {Rigid} {Body} into a {Viscous} {Incompressible}
  {Fluid}.
\newblock {\em Archive for Rational Mechanics and Analysis}, 223(3):1307--1335,
  2017.

\bibitem{lions_quelques_1969}
J.-L. Lions.
\newblock {\em Quelques méthodes de résolution des problèmes aux limites non
  linéaires}.
\newblock Dunod, 1969.

\bibitem{san_martin_global_2002}
J.~San~Martín, V.~Starovoitov, and M.~Tucsnak.
\newblock Global {Weak} {Solutions} for the {Two}-{Dimensional} {Motion} of
  {Several} {Rigid} {Bodies} in an {Incompressible} {Viscous} {Fluid}.
\newblock {\em Archive for Rational Mechanics and Analysis}, 161(2):113--147,
  2002.

\bibitem{serre_chute_1987}
D.~Serre.
\newblock Chute {Libre} d'un {Solide} dans un {Fluide} {Visqueux}
  lncompressible. {Existence}.
\newblock {\em Japan Journal of Industrial and Applied Mathematics},
  4(1):99--110, 1987.

\bibitem{takahashi_global_2004}
T.~Takahashi and M.~Tucsnak.
\newblock Global strong solutions for the two-dimensional motion of an infinite
  cylinder in a viscous fluid.
\newblock {\em Journal of Mathematical Fluid Mechanics}, 6(1):53--77, 2004.

\bibitem{yudakov_solvability_1974}
N.~V. Yudakov.
\newblock The solvability of the problem of the motion of a rigid body in a
  viscous incompressible fluid.
\newblock {\em Dinamika Sploshnoi Sredy}, 18:249--253, 1974.

\end{thebibliography}
\end{document}